\documentclass[12pt]{amsart}
\usepackage{amssymb,amscd}
\usepackage{verbatim}

\usepackage{amsmath,amssymb,graphicx,mathrsfs}   
\usepackage{enumerate}
\usepackage[colorlinks=true,allcolors = blue]{hyperref} 

\usepackage{tikz}
\usetikzlibrary{matrix}

\usepackage[all]{xy}

\let\frak\mathfrak

\def\>{\relax\ifmmode\mskip.666667\thinmuskip\relax\else\kern.111111em\fi}
\def\<{\relax\ifmmode\mskip-.333333\thinmuskip\relax\else\kern-.0555556em\fi}
\def\vsk#1>{\vskip#1\baselineskip}
\def\vv#1>{\vadjust{\vsk#1>}\ignorespaces}
\def\vvn#1>{\vadjust{\nobreak\vsk#1>\nobreak}\ignorespaces}

  \let\ssize\scriptstyle
\let\sssize\scriptscriptstyle

\let\Medskip\medskip
\def\medskip{\par\Medskip}
\let\Bigskip\bigskip
\def\bigskip{\par\Bigskip}

\let\Maketitle\maketitle
\def\maketitle{\Maketitle\thispagestyle{empty}\let\maketitle\empty}

\newtheorem{thm}{Theorem}[section]
\newtheorem{cor}[thm]{Corollary}
\newtheorem{lem}[thm]{Lemma}

\newtheorem{ex}[thm]{Example}
\newtheorem{exmp}[thm]{Example}

\theoremstyle{definition}                                  
\numberwithin{equation}{section}

\theoremstyle{definition}

\let\mc\mathcal
\let\nc\newcommand

\let\al\alpha

\let\la\lambda

\let\phi\varphi
\let\si\sigma
\let\Si\Sigma

\let\Om\Omega

\let\der\partial

\let\ox\otimes

\let\geq\geqslant

\let\leq\leqslant

\let\on\operatorname
\let\bi\bibitem
\let\bs\boldsymbol

\def\C{{\mathbb C}}
\def\Z{{\mathbb Z}}

\def\F{{\mathbb F}}   

\def\+#1{^{\{#1\}}}

\def\gl{\mathfrak{gl}}

\def\beq{\begin{equation}}
\def\eeq{\end{equation}}
\def\be{\begin{equation*}}
\def\ee{\end{equation*}}

\nc{\bea}{\begin{eqnarray*}}
\nc{\eea}{\end{eqnarray*}}
\nc{\bean}{\begin{eqnarray}}
\nc{\eean}{\end{eqnarray}}

\let\Ga\Gamma

\nc{\Il}{{\mc I_{\bs\la}}}
\nc{\bla}{{\bs\la}}
\nc{\Fla}{\F_\bla}
\nc{\tfl}{{T^*\Fla}}
\nc{\GL}{{GL_n(\C)}}
\nc{\GLC}{{GL_n(\C)\times\C^*}}

\let\sd s 

\def\ddk_#1{\kk_{#1}\<\>\frac\der{\der\<\>\kk_{#1}}}

\def\bul{\mathbin{\raise.2ex\hbox{$\sssize\bullet$}}}
\def\intt{\mathchoice
{\mathop{\raise.2ex\rlap{$\,\,\ssize\backslash$}{\intop}}\nolimits}
{\mathop{\raise.3ex\rlap{$\,\sssize\backslash$}{\intop}}\nolimits}
{\mathop{\raise.1ex\rlap{$\sssize\>\backslash$}{\intop}}\nolimits}
{\mathop{\rlap{$\sssize\<\>\backslash$}{\intop}}\nolimits}}

\let\kk q 
\let\cc c

\let\Ko K

\def\GZ/{Gelfand-Zetlin}
\def\KZ/{{\slshape KZ\/}}
\def\qKZ/{{\slshape qKZ\/}}
\def\XXX/{{\slshape XXX\/}}

\nc{\A}{{\mc A}}

\def\Sing{{\on{Sing\,}}}

\def\slt{{\frak{sl}_2}}

\nc{\hsl}{\widehat{{\frak{sl}_2}}}

\nc{\BC}{{ \mathbb C}}
\nc{\lra}{\longrightarrow}
\nc{\CO}{{\mathcal{O}}}
\nc{\BZ}{{ \mathbb Z}}
\nc{\hfn}{\hat{\frak{n}}}
\nc\Zs{{\Z/p^s\Z}}
\nc\Zo{{\Zs[z]^0}}
\nc\gr{{\on{gr}}}

\nc\fD{{\frak D}}

\begin{document}

\hrule width0pt
\vsk->

\title[Frobenius-like structure in Gaudin model]
{Frobenius-like structure in Gaudin model}

\author[Evgeny Mukhin and Alexander Varchenko]
{Evgeny Mukhin$\>^\circ$ and Alexander Varchenko$\>^\star$}

\maketitle

\begin{center}

{\it $\kern-.4em^\circ\<$Department of Mathematical Sciences,
Indiana University\,--\>Purdue University Indianapolis\kern-.4em\\
402 North Blackford St, Indianapolis, IN 46202-3216, USA\/}

\vsk.5>

{\it $^{\star}\<$Department of Mathematics, University
of North Carolina at Chapel Hill\\ Chapel Hill, NC 27599-3250, USA\/}

\vsk.5>
{\it $^{ \star}$ Faculty of Mathematics and Mechanics, Lomonosov Moscow State
University\\ Leninskiye Gory 1, 119991 Moscow GSP-1, Russia\/}

\end{center}

{\let\thefootnote\relax
\footnotetext{\vsk-.8>\noindent
$^\circ\<${\sl E\>-mail}:\enspace emukhin@iupui.edu\>,
supported in part by Simons Foundation grants \rlap{353831, 709444}
\\
$^\star\<${\sl E\>-mail}:\enspace anv@email.unc.edu\>,
supported in part by NSF grant  DMS-1954266}}

\begin{abstract}
We introduce a Frobenius-like structure for the $\slt$ Gaudin model.
Namely, we introduce potential functions  of the first and second kind. We describe 
the Shapovalov form in terms of derivatives of the potential of the first kind 
and the action of Gaudin Hamiltonians in terms of derivatives of the potential of the second kind.

\end{abstract}

\vsk>
{\leftskip3pc \rightskip\leftskip \parindent0pt \Small
{\it Key words\/}:  Gaudin Hamiltonians; Shapovalov form; potentials of the first and second kind.

\vsk.6>
{\it 2020 Mathematics Subject Classification\/}: 81R22 (53D45, 32S22)
\par}



\setcounter{footnote}{0}
\renewcommand{\thefootnote}{\arabic{footnote}}

\section{Introduction}

Frobenius manifolds were introduced by B. Dubrovin in the study of 
topological field theories, \cite{D1}. Frobenius manifolds is an important ingredient of the theory of integrable systems.

\vsk.2>

A Frobenius algebra is a commutative algebra $A$  with a nondegenerate bilinear form $(\,,\,)$ such that
$(uv,w)=(v,uw)$,\,$\forall u,v,w\in A$.

Roughly speaking, a Frobenius manifold is a manifold $M$ with a flat metric 
$ (\,,\,)$ and a Frobenius algebra structure on tangent spaces $T_xM$ at points $x\in M$ 
such that the structure constants of multiplication are given by the third derivatives
of a potential function with respect to flat coordinates. 
More precisely, let $z_1,\dots,z_n$ be local
coordinates on $M$ in which the metric is constant, then 
\begin{equation}\label{F}
\left(\frac{\der}{\der z_i}\cdot \frac{\der}{\der z_j},\
\frac{\der}{\der z_k}\right)_{\!\!x}\  =\   \frac{\der^3 L}{\der z_i\der z_j\der z_k}(x)
\end{equation}
for a suitable potential function $L$ on the manifold.
Formula \eqref{F} is a remarkable way to pack all information about this family of Frobenius algebras into one function.
 
 \vsk.2>
 
 A source of families of  Frobenius algebras is quantum cohomology algebras
of algebraic  varieties. Algebras of such a family depend on quantum parameters and 
form a Frobenius manifold, in which  the bilinear form $(\,,\,)$ is 
the intersection form on the corresponding variety and the potential function is defined in 
terms of enumerative geometry of curves on the 
variety, see \cite{D1}.

\vsk.2>

Another  source of families of Frobenius algebras is quantum integrable models related 
to representation theory. 
In this case, one starts with a tensor product of evaluation representations of some algebra
 (like the universal enveloping algebra of a current algebra, or a Yangian, 
 or a quantum affine algebra), which has a large commutative subalgebra called the 
Bethe subalgebra. Then the representation itself depends on the corresponding evaluation parameters,
while the image of the Bethe subalgebra in the representation 
is often a Frobenius algebra with respect to  the corresponding
Shapovalov form, see for example \cite{MTV1, L}. In this note we discuss the problem, posed in \cite{PV},
if this family of Frobenius algebras depending
on evaluation parameters has glimpses  of a Frobenius structure.

\vsk.2>

We study  the simplest example of the $\slt$ Gaudin model on a tensor product of vector representations. 
The Bethe algebra acts in the space 
$(\C^2)^{\otimes n}$ by operators 
depending on evaluation parameters $z_m$, $m=1,\dots,n$. The operators 
 are symmetric with respect to the Shapovalov form and
  commute with the diagonal action of $\slt$. The Bethe 
  algebra is generated by the Gaudin Hamiltonians $H_m(z_1,\dots,z_n)$, $m=1,\dots,n,$ 
  and the identity operator. We concentrate on the space $\Sing (\C^2)^{\otimes n}[n-2k]$ 
  of singular vectors of weight $n-2k$,
  which is  invariant with respect to the Bethe algebra.
This space is known to be cyclic with respect to the action of Gaudin Hamiltonians, 
and the algebra of Gaudin Hamiltonians is Frobenius, see \cite{MTV1}.

\vsk.2>

For $k>1$, one cannot expect 
formula \eqref{F} to be literally satisfied,
 as the number of evaluation parameters in our family of algebras is smaller
  than the dimension of the algebra, but it turns out that there is an analogous formula which looks as follows.
  
We have a natural spanning set  of vectors $\{v_I\}\subset \Sing (\C^2)^{\otimes n}[n-2k]$ labeled by 
$k$-element subsets $I\subset\{1,\dots,n\}$. The vectors $\{v_I\}$ are 
orthogonal projections of the standard tensor basis  of $(\C^2)^{\otimes n}[n-2k]$. 
We present two potential functions $P$ and $Q$ depending on $nk$
variables $z_i^{(j)}$,  $i=1,\dots, n$, $j=1,\dots,k$. We also introduce
 differential operators $\partial_I$ of order $k$, which involve derivatives
  with respect to the variables $z_i^{(j)}$, $i\in I$. Then the following holds.

\begin{itemize}
\item
The function $P$ is a polynomial of degree $2k$  written as a sum, 
where up to a common constant, each term is a product of $k$ factors of the form $(z_i^{(j)}-z_l^{(j)})^2$,
see \eqref{1st k}.
 It has the property 
\bean
\label{2}
(v_I,v_J)=\partial_I\partial_J P \qquad\forall I,J.
\eean

\item
The function $Q$ is a sum of terms of the form
 $\ln(z_i^{(1)}-z_s^{(1)})\,(z_i^{(1)}-z_l^{(1)})^2\, p $, where $(z_i^{(1)}-z_l^{(1)})^2\, p$ is a term of $P$,
 see \eqref{2nd k}. It has the property 
\bean
\label{3}
(H_m(z_1^{(1)}, \dots, z_n^{(1)})v_I,v_J)=\frac{\partial}{\partial z_m^{(1)}}\partial_I\partial_J Q
\qquad\forall I,J\ \on{and}\ m=1,\dots,n.
\eean
\end{itemize}

We call the functions $P$ and $Q$ the potentials of the first and second kind, respectively.
Formula \eqref{2} describes 
the Shapovalov form in terms of derivatives of the potential of the first kind, 
and formula \eqref{3} describes the action of Gaudin Hamiltonians in terms of derivatives of the potential of the second kind.

The existence of a polynomial $P$ satisfying \eqref{2}
is obvious, but the form of the answer seems to be interesting. 
The reason for existence of $Q$  is not clear. It looks interesting that the potential of the 
first kind describing only the Shapovalov form is so closely related to the potential of the 
second kind which describes the action of the Gaudin Hamiltonians.

\vsk.2>

Our construction is motivated by \cite{V3, PV},  where a
 Frobenius-like structure was introduced for a family of weighted 
 hyperplane arrangements, in which every hyperplane independently 
 moves parallelly to itself, see Theorem \ref{thm PV} below.
It is well-known that the Gaudin model on $\Sing W^{\ox n}[n-2k]$ is related to a
certain family of weighted discriminantal hyperplane arrangements
in $\C^k$ with hyperplanes depending on $n$ parameters $z_1,\dots,z_n$, 
see \cite{SV, V1, V2, TV}. These arrangements have $nk + k(k-1)/2$
hyperplanes and are symmetric with respect to permutations of coordinates in $\C^k$. 
 While this family of arrangements does not satisfy the assumptions of Theorem \ref{thm PV}, we get enough 
insight to construct the potentials $P$ and $Q$.

\medskip

We expect that potentials of the
 first and second kind exist for spaces of 
singular  vectors in tensor products of $\frak{sl}_N$ vector representations, 
where the Bethe algebra is still generated by Gaudin Hamiltonians, see \cite{MTV2}.

\medskip

In Section \ref{sec1a} we recall Frobenius-like structures related to arrangements of hyperplanes.
In Section \ref{sec2} we collect preliminary information. In Section \ref{sec3} 
we introduce potentials and relate them to the Gaudin model.

\section{Crtitical points and arrangements of hyperplanes}\label{sec1a}
Algebras of functions on critical sets of functions produce families of Frobenius algebras as follows.  
Let $\Phi(t_1,\dots,t_k)$ be a holomorphic function on an open set $D\subset \C^k$ with finitely many critical points $q\in D$,  
\bea
\frac{\der \Phi}{\der t_i}(q) =0, \quad i=1,\dots,k.
\eea
One considers the finite-dimensional { algebra of functions on the critical set},
\bea
A  = \mc O(D)\Big/\Big(\frac{\der\Phi}{\der t_j}\, \Big|\,
j=1,...,k\Big).
\eea
The Grothendieck residue  defines a nondegenerate  bilinear form $(\,,\,)$ on 
$A$,
 \bea
([f], [g]) =\frac{1}{(2\pi i)^k}\int_{\Gamma}
\frac{fg\ dt_1\wedge\dots\wedge dt_k}{\prod_{j=1}^k \frac{\der \Phi}{\der t_j}}, \quad [f], [g]\in A,
\eea
where $\Ga=\big\{t\in D\ \big| \ \ \big|  \frac{\der\Phi}{\der t_j}  \big|=\epsilon, j=1,\dots,k\big\}$.
The algebra $(A,(\,,\,))$ is a Frobenius algebra. 
In singularity theory this algebra is called the Milnor algebra.

\vsk.2>

This algebra of functions is especially interesting in the case  when the starting
 function $\Phi$ is the master function of an arrangement of hyperplanes.
On the one hand, the arrangements of hyperplanes lead to a simpler, more combinatorial setting. 
On the other hand, such algebras
are known to be related to the algebras coming from quantum integrable systems and quantum cohomology,
see for example \cite{SV, V1, V2, MTV1, GRTV}.

It turns out that  the algebra of functions on the critical set of  a master function of 
an arrangement has a Frobenius-like structure,  which is determined by two potentials.

\vsk.2>

Consider $\C^k$ with coordinates $t_1,\dots,t_k$ and an arrangement
$\mc C(z)$ of $n$ hyperplanes in $\C^k$ depending on parameters
$z=(z_1,\dots,z_n)$. The hyperplanes $H_i(z_i)$ of the arrangement
are defined by equations, $f_i(t,z) = {\sum}_{j=1}^k b^j_it_j+z_i = 0$, where $ b^j_i\in\C$ are fixed.
If  $z_i$ changes, the hyperplane $H_i(z_i)$ moves parallelly to itself.

Fix positive numbers $a=(a_1,...,a_n)$  called the weights.
The master function of the weighted arrangement $(\mc C(z),a)$
is the function
\bea
\Phi(t,z) =\sum_{i=1}^n\, a_i\log f_i(t,z).
\eea
Denote $U_z = \C^k-\mc C(z)$ the complement to $\mc C(z)$ and 
$\mc O(U_z)$ the algebra of regular functions on the complement $U_z$.
Denote
\bea
A_z  = \mc O(U_z)\Big/\Big(\frac{\der\Phi}{\der t_j}\, \Big|\,
j=1,...,k\Big)
\eea
the algebra of functions on the critical set of the master function $\Phi(\cdot,z)$
 restricted to the complement $U_z$. 
  Denote $(\,,\,)_z$ the Grothiendick residue form on  $A_z$.

In the space $\C^n$ of parameters  $z$ there is a hypersurface $\Si$, called 
the discriminant,  characterized by the property:
if $z\in\C^n-\Si$, then the arrangement $\mc C(z)$ has normal crossings only.
We may compare the algebras $A_z$ for $z\in\C^n-\Si$ as follows.

\vsk.2>

For $i=1,\dots,n$ the elements
$p_i =  \big[\frac{\der \Phi}{\der z_i}\big] = \big[\frac {a_i}{f_i}\big] \in A_z\,$
 generate $A_z$ as an algebra. 
We say that a subset $\{i_1,\dots,i_k\}\subset\{1,\dots,n\}$ is independent if
$d_{i_1,\dots,i_k}=\det (b^j_{i_\ell})_{\ell,j=1}^k\ne 0$.
For $z\in\C^n-\Si$  the elements $d_{i_1,\dots,i_k} p_{i_1} \cdots p_{i_k}$ span  $A_z$ as a vector space.
  
The defining linear relations between the elements $d_{i_1,\dots,i_k} p_{i_1} \cdots p_{i_k}$
are labeled by 
$(k-1)$-element subsets $\{i_1,\dots,i_{k-1}\}\subset\{1,\dots,n\}$,
\bea
\sum_{m=1}^n \,d_{i_1,\dots,i_{k-1},m}\,p_{m}\cdot  p_{i_1} \cdots p_{i_{k-1}} \,=\, 0.
\eea
Hence   the dimension $\dim_\C A_z$ does not depend on $z\in\C^n-\Si$.

Consider the  complex vector bundle 
\bea
A={\cup}_{z\in\C^n-\Si}A_z   \to \C^n-\Si 
\eea 
whose fiber over a point $z\in\C^n-\Si$ is the Frobenius algebra $A_z$.  
Identifying the elements $d_{i_1,\dots,i_k} p_{i_1} \dots p_{i_k}$ in all fibers we {trivialize} the bundle.

\begin{thm} [\cite{V3, PV}]
\label{thm PV}
  There exist two functions $P$, $Q$ on $\C^n-\Si$, called the 
{ potentials of the first and second kind}, with the following properties.
For any two  independent sets $\{i_1,\dots,i_k\}, \{j_1,\dots,j_k\}$ and any
index  $m=1,\dots,n$, we have
\bea
&&
(p_{i_1}\cdots p_{i_k}, p_{j_1}\cdots p_{j_k})_z
=
 \frac{\der^{2k}P}{\der z_{i_1}\dots\der z_{i_k}\der z_{j_1}\dots\der z_{j_k}}(z),
\\
&&
(p_{m}\cdot p_{i_1}\cdots p_{i_k}, p_{j_1}\cdots p_{j_k})_z
=
 \frac{\der^{2k+1}Q}{\der z_{m}\der z_{i_1}\dots\der z_{i_k}\der z_{j_1}\dots\der z_{j_k}}(z).
\eea
The potentials are given by some combinatorial formulas.

The potential $P$ of the first kind is a polynomial of degree $2k$ and hence
all
\\
 $(p_{i_1}\cdots p_{i_k}, p_{j_1}\cdots p_{j_k})_z$ are constants.

\end{thm}

The first formula determines the Grothendieck residue bilinear form 
$(\,,\,)_z$ in terms of the potential of  the first kind. The second formula
determines the operators of multiplication by multiplicative generators 
$\{p_j\}$, $j=1,\dots,n$, in terms of the potential of the  second kind.

\vsk.2>

This pair of potentials is called 
in \cite{V3, PV} a Frobenius-like structure associated with the family 
$(\mc C(z),a)$ of weighted arrangements in $\C^k$.

\begin{exmp}
[\cite{V3}]
\label{ex1}

For $k=1$ consider the arrangement of $n$ points on line
defined by equations $t+z_i=0$, $i=1,\dots,n$, with weights $a_1,\dots,a_n$.
Then
\bea
&
P(z) = \frac1{a_1+\dots+a_n} \sum_{1\leq i< j\leq n} a_ia_j\frac{(z_i-z_j)^2}2,
\qquad
Q(z) = \sum_{1\leq i< j\leq n} a_ia_j\,\ln(z_i-z_j)\,\frac{(z_i-z_j)^2}{2}\,,
\\
&
(p_i,p_j)_z =  \frac{\der^{2}P}{\der z_i\der z_j}(z),
\qquad
(p_m\cdot p_i,\, p_j)_z  = \frac{\der^{3}P}{\der z_h\der z_i\der z_j}(z).
\eea
If $a_1=\dots=a_n$, then this Frobenius-like structure is the almost dual
 Frobenius structure associated with the Weyl group $W(A_{n-1})$ in \cite{D2}.

\end{exmp}

\begin{exmp} 
[\cite{PV}]
\label{ex2}

 For the arrangement of four lines  on plane given by equations
$t_2+        z_1=0$,
$  t_2+        z_2=0$,
$   t_1+z_3=0$,  $  t_1+t_2+z_4=0$ we have 
\bea
&&
P = \frac1{a_1+a_2+a_3+a_4}
\Big(a_1a_3a_4 \frac{(z_1+z_3-z_4)^4}{4!}
+ a_2a_3a_4 \frac{(z_2+z_3-z_4)^4}{4!}
\\
\notag
&&
\phantom{aaaaaaaaaaaaaaaaaaa}
+ \frac{a_1a_2a_3a_4}{a_3+a_4} 
 \frac {(z_1-z_2)^2}{2!}\frac{(z_1+z_3-z_4)^2}{2!} \Big),
\eea
\bea
&&
Q =
a_1a_3a_4 \ln(z_1+z_3-z_4)\frac{(z_1+z_3-z_4)^4}{4!}
+ a_2a_3a_4\ln(z_2+z_3-z_4) \frac{(z_2+z_3-z_4)^4}{4!}
\\
\notag
&&
\phantom{aaaaaaaaaaaaaa}
+ \frac{a_1a_2a_3a_4}{a_3+a_4} 
 \ln(z_1-z_2)\frac {(z_1-z_2)^2}{2!}\frac{(z_1+z_3-z_4)^2}{2!} .
\eea
Theorem \ref{thm PV}  in particular says that $(p_1p_3, p_2p_4)_z = \frac 
{a_1a_2a_3a_4}
{(a_1+a_2+a_3+a_4)(a_3+a_4)}$ and this does not depend on $z\in\C^n-\Delta$, and
$(p_4p_1p_3,p_3p_4)_z$ 
$= 
\frac{a_1a_3a_4}{z_1+z_3-z_4}$.

\end{exmp}

In this example
 the potentials are sums of terms corresponding to subarrangements 
consisting of three or four lines, corresponding to triangles and trapezoids.
It turns out that this is the general case.  
One  introduces  the notion of an elementary arrangement in $\C^k$.
The elementary subarrangements in $\C^2$ are triangles 
and trapezoids.
An elementary arrangement in $\C^k$
has at most $2k$ hyperplanes. 
The potentials  are sums, over all elementary subarrangements, of 
some explicit prepotentials of the elementary subarrangements,
see \cite{V3, PV}

The fact that the potentials are sums of contributions from elementary 
subarrangements indicates a phenomenon of
 {\it locality} of 
 Grothendieck residue bilinear form and multiplication on the algebra $A_z$.
We observe a similar locality property  in the Gaudin model potentials,
see formulas \eqref{1st k} and \eqref{2nd k}.

On Frobenius-like structures see also \cite{HV,V4}.

\vsk.2>

\section{Shapovalov form}
\label{sec2}

Let $n, k$ be positive integers.

\subsection{Space of singular vectors}

Consider the complex Lie algebra $\slt$ with generators $e,f,h$ and relations $[e,f]=h,\, [h,e]=2e, \,[h,f]=-2f$.
Consider the complex vector space $W$ with basis $w_1,w_2$ and the $\slt$-action,
\bea
e=\begin{pmatrix}
0 & 1
\\
0 & 0 & 
\end{pmatrix}, \qquad
f=\begin{pmatrix}
0 & 0
\\
1 & 0 & 
\end{pmatrix},
\qquad
h=\begin{pmatrix}
1 & 0
\\
0 & -1 & 
\end{pmatrix}.
\eea 
The $\slt$-module $W^{\ox n}$ has a basis labeled by subsets $I\subset\{1,\dots,n\}$,
\bea
V_I= w_{i_1}\ox \dots\ox w_{i_n},
\eea
where $i_j =1$ if $i_j\notin I$  and
$i_j =2$ if $i_j\in I$.

\vsk.2>

The Shapovalov form $(\,,\,)$ on $W^{\ox n}$ is the symmetric bilinear form such that
$(V_I,V_J)=\delta_{IJ}$.  It has the properties:\, $(hx,y)=(x,hy)$, $(ex,y)=(x,fy)$ for all $x,y\in W$.

\vsk.2>

Consider the weight decomposition of $W^{\ox n}$ 
into eigenspaces of $h$,  $W^{\ox n} = \sum_{k=0}^n W^{\ox n}[n-2k]$. 
The vectors $V_I$ with $|I|=k$ form a  basis of $W^{\ox n}[n-2k]$.
Define the space of singular vectors of weight $n-2k$,
\bea
\Sing W^{\ox n}[n-2k] = \{ w\in W^{\ox n}[n-2k]\mid ew=0\}.
\eea
This space is nonempty if and only if $n\geq 2k$. We assume that $n\geq 2k$.

\subsection{Orthogonal projection}

 Let
\bean
\label{pi}
\pi : W^{\ox n}[n-2k] \to \Sing W^{\ox n}[n-2k]
\eean 
be the orthogonal projection with respect to the Shapovalov form. 
The kernel of the projection is the image of the operator
\bea
f : W^{\ox n}[n-2(k-1)] \to W^{\ox n}[n-2k].
\eea
For a $k$-element subset $I\subset \{1,\dots,n\}$
denote
\bean
\label{vI}
v_{I}= \pi (V_{I}) \in \Sing W^{\ox n}[n-2k].
\eean

There are $\binom{n}{k}$ vectors $v_I$ in the $\binom{n}{k}- \binom{n}{k-1}$-dimensional space
$\Sing W^{\ox n}[n-2k]$. The defining linear relations between these vectors are labeled by $(k-1)$-element subsets
$K$ and have the form
\bea
\sum_{m\not\in K} v_{K\cup \{m\}}=0.
\eea

\vsk.2>

Clearly the vector $v_I$  has the form
\bea
v_{I} = V_I + b_{k-1} {\sum}_{k-1} fV_K + b_{k-2} {\sum}_{k-2} fV_K + \dots + b_0 {\sum}_0 fV_K,
\eea
where $b_\ell$ are suitable numbers, and $\sum_\ell$ is the sum over all $(k-1)$-element subsets $K$ such that $|I\cap K|= \ell$.

\begin{lem}
We have
\bean
\label{b fo}
\phantom{aaaa}
b_\ell\,=\,(-1)^{k-\ell}\,\frac{(k-\ell-1)!}{\prod_{i=\ell+1}^{k} (n-2k+1+i)}\,
=\,(-1)^{k-\ell}\,\frac{(k-\ell-1)!(n-2k+\ell+1)!}{(n-k+1)!}\,,
\eean
$\ell = 0, \dots, k-1$.

\end{lem}

\begin{proof}
The property $e v_I =0$ produces the following system of equations:
\bea
&
1+b_{k-1}n+b_{k-2}(k-1)(n-k) = 0,
\\
\notag
&
b_{\ell+1} (k-\ell)(k-\ell-1) + b_\ell(k-\ell)(n-2k+2\ell+2) + b_{\ell-1}\ell(n-2k+\ell+1) =0
\eea
for $\ell=0,\dots , k-2$. This system implies
\bea
b_r\,
= \,b_0\,(-1)^r\,\prod_{i=1}^r\frac{n-2k+1+i}{k-i}\,,
\qquad
b_0\,=\,
(-1)^k\,\frac{(k-1)!}{\prod_{i=1}^k (n-2k+1+i)}\,,
\eea
and hence \eqref{b fo}.
\end{proof}

Clearly  $v_I$ has the form:
\bea
v_{I} = a_k V_I + a_{k-1} {\sum}^{k-1} V_J + a_{k-2} {\sum}^{k-2} V_J + \dots + a_0 {\sum}^0 V_J,
\eea
where $a_\ell$ are suitable numbers, and $\sum^\ell$ is the sum over all $k$-element subsets $J$ such that $|I\cap J|= \ell$.

\begin{lem}

We have
\bean
\label{aell}
\phantom{aaaaaa}
a_\ell  =    
 (-1)^{k+\ell} \, \frac{(n-2k+1)\,(k-\ell)!}{\prod_{i=\ell}^k (n-2k+1+i)}
\,=\, (-1)^{k+\ell} \,\frac{(n-2k+1)\,(k-\ell)!\,(n-2k+\ell)!}{(n-k+1)!}\,,
\eean
$\ell = 0,\dots, k.$

\end{lem}

\begin{proof}

Let $K$ be a $(k-1)$-element subset. The condition
$(fV_K,v_I) =0$\
produces the following system of equations. If $|K\cap I| = \ell$, where $\ell = 0, 1, \dots,k-1$, then
\bea
(k-\ell) a_{\ell+1} + (n-2k+\ell+1) a_\ell = 0.
\eea
This system has the solution 
\bean
\label{ai}
a_\ell=(-1)^\ell\frac{(n-2k+1)(n-2k+2)\dots (n-2k+\ell)}{k(k-1)\dots(k-\ell+1)}\,a_0\,.
\eean
It is easy to see that $a_0 =kb_0$. Thus
\bea
a_0 = (-1)^k\,\frac{k!}{\prod_{i=1}^k (n-2k+1+i)}\,,
\eea
and formula \eqref{aell} for any $\ell$ follows from formula \eqref{ai}.
\end{proof}

\begin{ex}  
For $k=1$ we have
$b_0\,=\,-\,\frac1n\,,\,
a_0\,=\, -\,\frac1n\,,\,
a_1\,=\,\frac{n-1}n\,.$\,
For  $k=2$ we have
\bea
&
b_0\ =\ \frac 1{(n-1)(n-2)}\,,
\qquad
 b_1\ =\ -\,\frac 1{n-1}\,,\,
 \\
&
a_0=\frac{2}{(n-2)(n-1)}\,,
\qquad
a_1=-\frac{n-3}{(n-2)(n-1)}\,,
\qquad
a_2= \frac{n-3}{n-1}\,.
\eea
\end{ex}

\begin{lem}
Let $I,J$ be two  $k$-element subsets such that $|I\cap J|=\ell$. Then
\bean
\label{SIJ}
(v_I,v_J) = (v_I,V_J)=a_\ell\,.
\eean
\qed

\end{lem}

The following formulas are useful:
\bean
\label{a-a}
a_{\ell-1}-a_{\ell}  
&=&
(-1)^{k-\ell+1} \, \frac{(n-2k+1)\,(k-\ell)!}{\prod_{i=\ell-1}^{k-1}(n-2k+1+i)}\,
\\
\notag
 &=&
(-1)^{k-\ell+1} \, \frac{ (n-2k+1)\,(k-\ell)!\,(n-2k+\ell-1)!}{(n-k)!}\,
\eean
$\ell=1,\dots,k.$

\subsection{Remark}
 One may show that the orthogonal projection 
$\pi : W^{\ox n}[n-2k] \to 
\Sing W^{\ox n}[n-2k]$ has the following locality property
with respect to the number $n$ of tensor factors.

\vsk.2>

Consider the orthogonal projection
$\pi_{2k} : W^{\ox 2k}[0] \to 
\Sing W^{\ox 2k}[0]$ and the image
$\pi_{2k} V^{2k}_{\{1,\dots,k\}}$ 
of the vector
\bea
V^{2k}_{\{1,\dots,k\}} = w_2\ox \dots\ox w_2\ox w_1\ox \dots\ox w_1 \in W^{\ox 2k}[0] .
\eea
For any $k$-element subset $J \subset \{k+1,\dots,n\}$, let
$(\pi_{2k} V^{2k}_{\{1,\dots,k\}})^{(J)} \in \Sing W^{\ox n}[n-2k]$  be the vector 
$\pi_{2k} V^{2k}_{\{1,\dots,k\}}$ placed in the tensor factors of $W^{\ox n}$ with
indices $\{1,\dots,k\}\cup J \subset \{1,\dots,n\}$, and  with the vectors $w_1$ staying in other factors of $W^{\ox n}$. 
Then the vector $v_{\{1,\dots,k\}} \in \Sing W^{\ox n}[n-2k]$, defined in \eqref{vI}, equals the sum $\sum_J
(\pi_{2k} V^{2k}_{\{1,\dots,k\}})^{(J)}$ up to a multiplicative constant.

For example, for $k=1$ we have
\bea
&&
v_{\{1\}} = \frac{n-1}n V_{\{1\}} - \sum_{j=2}^n\frac1nV_{\{j\}},
\qquad
\pi_{2} V^{2}_{\{1\}}= \frac{1}2 w_2\ox w_1 - \frac12w_1\ox w_2, 
\\
&&
v_{\{1\}} \,=\,\frac2n \,\sum_{j=2}^n \Big(\frac{1}2 V_{\{1\}} - \frac12V_{\{j\}}\Big).
\eea

The locality property of potential functions in formulas \eqref{1st k} and \eqref{2nd k} is of similar flavor.

\section{Potentials}
\label{sec3}

\subsection{Potential of the first kind}

Potentials of the first and second kind are function of $nk$ variables
\bea
z=
(z^{(1)}_1, \dots, z^{(1)}_n; z^{(2)}_1, \dots, z^{(2)}_n;
\dots; z^{(k)}_1, \dots, z^{(k)}_n).
\eea
For every $k$-element subset $I=(i_1,\dots, i_k)\subset \{1,\dots,n\}$
define the differential operator
\bea
\der_I \ =\  \sum_{\si\in S_k} \
\frac{\der^k}{\der z^{(1)}_{i_{\si_1}}\dots \der z^{(k)}_{i_{\si_k}}}\,.
\eea

Recall that $n\geq 2k$. Let $\al=(\{p_1,q_1\},\dots, \{p_k,q_k\})$ be a sequence of nonintersecting
unordered two-element subsets of $\{1,\dots,n\}$. Denote by $\mc A$ the set of all such sequences.
The number of elements in $\mc A$ is given by the formula:
\bea
|\mc A| = \binom{n}{2}\binom{n-2}{2}\cdots \binom{n-2k+2}{2} =\frac{n!}{2^k(n-2k)!}\,.
\eea

For every $\al=(\{p_1,q_1\},\dots, \{p_k,q_k\})\in \mc A$ define
\bean
\label{Pa}
P_\al(z) = \prod_{i=1}^k(z^{(i)}_{p_i}- z^{(i)}_{q_i})^2.
\eean
Define the {\it potential of the  first kind} $P(z)$ by the formula
\bean
\label{1st k}
\qquad
P(z) = c_1 \sum_{\al\in\mc A} P_\al(z)\,, 
\qquad
c_1= \frac 1{2^k\,k!\,\prod_{i=1}^k(n-2k+1+i)} = \frac{(n-2k+1)!}{2^k\,k!\,(n-k+1)!}\,.
\eean

\begin{thm}
\label{thm 1st}
For any two $k$-element subsets $I$ and
$ J$ we have
\bean
\label{mf1}
(v_I,v_J) = \der_I\der_J P(z)\,.
\eean

\end{thm}

\begin{proof}  It is enough to prove this formula for 
$I=\{1,2,\dots,k\}$ and
\\
$J=\{1,2,\dots,\ell, k+1,k+2,\dots, 2k-\ell\}$, 
where $\ell = 0, 1,\dots,k$. 
It is easy to see that in this case 
\bea
\der_I \der_J P \ =\  k!\,\frac{\der^k}{\der z^{(1)}_{1}\der z^{(2)}_{2}\dots \der z^{(k)}_k}
\cdot
\der_J P\,.
\eea
The number 
$\frac{\der^k}{\der z^{(1)}_{1}\der z^{(2)}_{2}\dots \der z^{(k)}_k}
\cdot
\der_J P_\al$\ 
is nonzero  only if $\al$ has the following form,
\bea
\al = (\{1,q_1\}, \{2,q_2\}, \dots, \{\ell,q_\ell\}, \{\ell+1, \si_{k+1}),
\{\ell+2, \si_{k+2}\}, \dots, \{k, \si_{2k-\ell}\}\},
\eea
where $(q_1, q_2, \dots, q_\ell)$ is an ordered subset of
$\{2k-\ell +1,2k-\ell +2, \dots, n\}$ and
\linebreak
$(\si_{k+1}, \si_{k+2},\dots,\si_{2k-\ell})$ is a permutation of
$(k+1,\dots,2k-\ell)$.

There are
$ (n-2k +\ell)(n-2k + \ell -1)\dots (n-2k+1)\,(k-\ell)!$ such sequences $\al$.
For every such $\al$, we have 
$\frac{\der^k}{\der z^{(1)}_{1}\der z^{(2)}_{2}\dots \der z^{(k)}_k}
\cdot\der_J P_\al = (-1)^{k-\ell} 2^k$.
To prove the theorem it is enough to check that
\bea 
a_\ell = c_1\, k!\, (-1)^{k-\ell} \,2^k \,(k-\ell)! \,\prod_{i=0}^{\ell-1}(n-2k+1+i). 
\eea
This follows from formulas \eqref{aell} and \eqref{1st k}.
\end{proof}

\subsection{Gaudin model} 
\label{sec G}

Define the Casimir element 
\bea
\Om 
=
 \frac12 h\ox h + e\ox f+f\ox e\ \  \in\ \ \slt\ox\slt\,.
 \eea
Define the linear operators on $W^{\ox n}$ depending on parameters $u=(u_1,\dots,u_n)$,
\bea
H_m(u) = \sum_{j=1\atop j\ne m}^n\frac{\Om_{mj}}{u_m-u_j}\,,
\qquad m=1,\dots,n,
\eea
where 
$\Om_{mj}:W^{\ox n}\to W^{\ox n}$ is the Casimir operator acting in the $m$th and $j$th tensor factors.
The operators $H_m(u)$ are called the Gaudin Hamiltonians.  The operators  are symmetric and commute,
\bea
(H_m(u) x,y)=(x,H_m(u) y) \ \ \forall x,y\in W^{\ox n},
\qquad
[H_m(u), H_j(u)]=0\ \ \forall \ m,j.
\eea
The operators commute with the $\slt$-action on $W^{\ox n}$ and hence preserve every 
$W^{\ox n}[n-2k]$ and   $\Sing W^{\ox n}[n-2k]$. The operators commute with the orthogonal projection $\pi$ defined in \eqref{pi},
$\pi H_m(u) = H_m(u) \pi$.

\vsk.2>
The algebra of endomorphisms of $\Sing W^{\ox n}[n-2k]$ generated by the scalar operators and the Gaudin Hamiltonians
is called the Bethe algebra of $\Sing W^{\ox n}[n-2k]$.

\vsk.2>

Introduce the reduced Casimir element and reduced Gaudin Hamiltonians,
\bea
\bar\Om 
=
 \frac12 h\ox h + e\ox f+f\ox e -\frac 12,
\qquad
\bar H_m(u) = \sum_{j=1\atop j\ne m}^n\frac{\bar\Om_{mj}}{u_m-u_j}\,,\quad m=1,\dots,n.
\eea
The reduced Gaudin Hamiltonians are symmetric, commute, and generate together with scalar operators
the same Bethe algebra as the Gaudin Hamiltonians.

\begin{lem}
Consider $\bar\Om$ as an linear operator on $W^{\ox 2}$. Then $\bar \Om = P-1$, where $P$ is the permutation of 
tensor factors and $1$ is the identity operator.
\qed
\end{lem}

Hence
\bean
\label{HVv}
\bar H_m(u) V_I 
&=&
 \sum_{j\notin I}\frac{-V_I + V_{I \cup \{j\} -\{m\}}}{u_m-u_j}, \qquad \text{if}\ m\in I,
\\
\notag
\bar H_m(u) V_I 
&=&
 \sum_{j\in I}\frac{-V_I + V_{I \cup \{m\} -\{j\}}}{u_m-u_j}, \qquad \text{if}\ m\notin I.
\eean
In  \eqref{HVv} the vectors $V_I,V_{I \cup \{j\} -\{m\}}, V_{I \cup \{m\} -\{j\}}$
can be replaced with
$v_I$,
$v_{I \cup \{j\} -\{m\}}$, $v_{I \cup \{m\} -\{j\}}$ since the operators $\bar H_m(u)$ commute with the orthogonal projection $\pi$.

\vsk.2>
The reduced Gaudin Hamiltonians are governed by a potential function of the second kind.

\subsection{Potential of the second kind}

For every $\al=(\{p_1,q_1\},\dots, \{p_k,q_k\})\in \mc A$ define
\bean
\label{Qa}
Q_\al(z) =\ln(z^{(1)}_{p_1}- z^{(1)}_{q_1}) \prod_{i=1}^k(z^{(i)}_{p_i}- z^{(i)}_{q_i})^2.
\eean
Define the {\it potential of the second kind} $Q(z)$ by the formula

\bean
\label{2nd k}
Q(z) 
&=& c_2 \sum_{\al\in\mc A} Q_\al(z)\,, 
\\
\notag
c_2 &=&  \frac {-1}{ 2^k\,(k-1)!\, \prod_{i=1}^{k-1}(n-2k+1+i)}
=
\frac {-\,(n-2k+1)!}{ 2^k\,(k-1)!\, (n-k)!}\,.
\eean

\vsk.2>

The following theorem describes the action of the reduced Gaudin Hamiltonians on 
the space $\Sing W^{\ox n}[n-2k]$ 
in terms of derivatives of the potential of the second kind.

\begin{thm}
\label{thm 2nd}
Let $I, J$ be two $k$-element subsets of $\{1,\dots,n\}$ and 
$m\in\{1,\dots,n\}$. Then
\bean
\label{Qfo}
(\bar H_m(z^{(1)}_1, \dots, z^{(1)}_n) v_I,v_J) \ =\  \frac{\der }{\der z_m^{(1)} } \der_I\der_J  Q(z)\,.
\eean

\end{thm}

\begin{proof}

We prove the theorem by an explicit computation based on the simple fact that  
any derivative of order $2k$  of any term $Q_\alpha$ in \eqref{Qa} with respect 
to various variables $z_i^{(j)}$ equals 
a constant multiple of $\ln(z^{(1)}_{p_1}- z^{(1)}_{q_1}) $   (which can be zero) 
plus a constant.

\vsk.2>

Assume that $m\notin I\cup J$  and $|I\cap J|=\ell$, $\ell=0,\dots,k$.
Under this assumption, without loss of generality we may assume that
$I=\{1,\dots,k\}$, $J=\{1,\dots,\ell, k+1,\dots,2k-\ell\}$, and $m=n$.
Then
\bea
(\bar H_n(z^{(1)}_1, \dots, z^{(1)}_n) v_I, v_J) 
&=& 
\sum_{i=1}^k \frac{(-v_{1,\dots,k}+v_{1,\dots,\widehat{i}, \dots, k, n}, 
v_{1,\dots,\ell, k+1,\dots,2k-\ell})}{z^{(1)}_n - z^{(1)}_i}
\\
&=&
\sum_{i=1}^\ell \frac{-a_\ell +a_{\ell-1}}{z^{(1)}_n - z^{(1)}_i} +
\sum_{i=\ell+1}^k \frac{-a_\ell +a_\ell}{z^{(1)}_n - z^{(1)}_i}=
 \sum_{i=1}^\ell \frac{-a_\ell + a_{\ell-1}}{z^{(1)}_n - z^{(1)}_i}\,.
\eea
Under this assumption, 
$ \frac{\der }{\der z_n^{(1)} }\der_{\{1,\dots,k\}}\der_{\{1,\dots,\ell, k+1,\dots,2k-\ell\}} Q_\al(z)$
  is nonzero only if 
 $\al$ has the following form,
$\al=(\{p_1,n\}, \{p_2,q_2\},\dots, \{p_k,q_k\})$,
where 
\begin{enumerate}
\item[(i)]

$p_1\in \{1,\dots,\ell\}$; 
\item[(ii)]
the sequence $(p_2,\dots, p_k)$ is a permutation of
the sequence $(1,\dots,\widehat{p_1}, \dots,k)$
\\
 (the sequence
$(p_2,\dots, p_k)$ determines a partition of $(2,\dots,k)$ into two subsequences:
$(i_1<\dots <i_{\ell-1})$ such that all $p_{i_r}\leq \ell$
and $(j_1<\dots <j_{k-\ell})$ such that all  $p_{j_r} > \ell$, these subsequences are used below);

\item[(iii)]
 the  sequence $(q_{i_1},\dots,q_{i_{\ell-1}})$  is an ordered $(\ell-1)$-element subset of
$\{2k-\ell+1, \dots$, $n-1\}$; 
\item[(iv)]
the sequence  $(q_{j_1},\dots,q_{j_{k-\ell}})$  is a permutation of
the sequence $(k+1,\dots,2k-\ell)$.

\end{enumerate}
The number of such $\al$ with fixed $p_1$ equals 
$(k-1)!\, (k-\ell)!\,\prod_{i=0}^{\ell-2}(n-2k+1+i)$.\
For every such $\al$, 
\bea
\frac{\der}{\der z_n^{(1)} }\der_{\{1,\dots,k\}}\der_{\{1,\dots,\ell, k+1,\dots,2k-\ell\}}  Q_\al(z)
 =  \frac{(-1)^{k-\ell}2^k}{z_n^{(1)}-z_{p_1}^{(1)}}.
\eea
It follows from \eqref{aell} and \eqref{2nd k} that
\bean
\label{alc}
\frac{a_{\ell-1}-a_\ell}{z_n^{(1)}-z_{p_1}^{(1)}} \ =\ \frac{(-1)^{k-\ell}2^k}{z_n^{(1)}-z_{p_1}^{(1)}}\,
c_2\,(k-1)!\, (k-\ell)!\,\prod_{i=0}^{\ell-2}(n-2k+1+i)\,,
\eean
and  \eqref{Qfo} holds in this case.

\vsk.2>

Assume that $m$ belongs to $I$ but not to $J$ and $|I\cap J|=\ell-1$,
$\ell =1,\dots,k$. 
Under this assumption, without loss of generality we may assume that
$I=\{1,\dots,k\}$, $J=\{2,\dots, \ell, k+1,\dots,2k-\ell+1\}$, and $m=1$.
Then
\bea
(\bar H_1(z^{(1)}_1, \dots, z^{(1)}_n) v_I, v_J) 
= 
\sum_{i=k+1}^{n} \frac{(-v_{1,\dots,k}+v_{2,\dots, k, i}, 
v_{2,\dots, \ell, k+1,\dots,2k-\ell+1})}{z^{(1)}_1 - z^{(1)}_i} 
=
\sum_{i=k+1}^{2k-\ell+1} \frac{-a_{\ell-1} +a_{\ell}}{z^{(1)}_1 - z^{(1)}_i}\,.
\eea
Under this assumption, 
$ \frac{\der}{\der z_1^{(1)} }\der_{\{1,\dots,k\}}\der_{\{2,\dots, \ell, k+1,\dots,2k-\ell+1\}} Q_\al(z)$
  is nonzero only if 
 $\al$ has the following form,
$\al=(\{1,q_1\}, \{p_2,q_2\},\dots, \{p_k,q_k\})$,
where
\begin{enumerate}
\item[(i)]

$q_1\in \{k+1,\dots,2k-\ell+1\}$; 
\item[(ii)]
the sequence $(p_2,\dots, p_k)$ is a permutation of
the sequence $(2,\dots,k)$
\\
 (the sequence
$(p_2,\dots, p_k)$ determines a partition of $(2,\dots,k)$ into two subsequences:
$(i_1<\dots <i_{\ell-1})$ such that all $p_{i_r}\leq \ell$
and $(j_1<\dots <j_{k-\ell})$ such that all  $p_{j_r} > \ell$,
these subsequences are used below);

\item[(iii)]
 the  sequence $(q_{i_1},\dots,q_{i_{\ell-1}})$  is an ordered $(\ell-1)$-element subset of
 $\{2k-\ell+2,\dots,n\}$;

\item[(iv)]
the sequence  $(q_{j_1},\dots,q_{j_{k-\ell}})$  is a permutation of
the sequence $(k+1,\dots,\widehat{q_1}, \dots, 
\linebreak
2k-\ell+1)$. 

\end{enumerate}
The number of such $\al$ with fixed $p_1$ equals 
$(k-1)!\, (k-\ell)!\,\prod_{i=0}^{\ell-2}(n-2k+1+i)$.
For every such $\al$, 
\bea
\frac{\der}{\der z_1^{(1)} }
\der_{\{1,\dots,k\}}\der_{\{2,\dots, \ell, k+1,\dots,2k-\ell+1\}} 
Q_\al (z) =  \frac{(-1)^{k-\ell}2^k}{z_{q_1}^{(1)}-z_{1}^{(1)}}\,.
\eea
It follows from \eqref{aell} and \eqref{2nd k} that
\bea
\frac{-a_{\ell-1} +a_{\ell}}{z^{(1)}_1 - z^{(1)}_{q_1}}\,=\,
 \frac{(-1)^{k-\ell}2^k}{z_{q_1}^{(1)}-z_{1}^{(1)}}\,
 c_2\,(k-1)!\, (k-\ell)!\,\prod_{i=0}^{\ell-2}(n-2k+1+i)\,,
\eea
and  \eqref{Qfo} holds in this case.

\vsk.2>

Assume that $m$ belongs to $I\cap J$  and $|I\cap J|=\ell$, $\ell = 1,\dots,k$.
Under this assumption, without loss of generality we may assume that
$I=\{1,\dots,k\}$, $J=\{1,\dots, \ell, k+1,\dots,2k-\ell\}$, and $m=1$.
Then
\bea
(\bar H_1(z^{(1)}_1, \dots, z^{(1)}_n) v_I, v_J) 
= 
\sum_{i=k+1}^{n} \frac{(-v_{1,\dots,k}+v_{2,\dots, k, i}, 
v_{1,\dots, \ell, k+1,\dots,2k-\ell})}{z^{(1)}_1 - z^{(1)}_i} 
=
\sum_{i=2k-\ell+1}^{n} \frac{-a_\ell +a_{\ell-1}}{z^{(1)}_1 - z^{(1)}_i}\,.
\eea
Under this assumption, 
$\frac{\der}{\der z_1^{(1)} }
\der_{\{1,\dots,k\}}\der_{\{ 1,\dots,\ell,k+1,\dots,2k-\ell\}}  Q_\al(z)$
  is nonzero only if 
 $\al$ has the following form,
$\al=(\{1,q_1\}, \{p_2,q_2\},\dots, \{p_k,q_k\})$
\begin{enumerate}
\item[(i)]

$q_1\in \{2k-\ell+1, \dots,n\}$; 
\item[(ii)]
the sequence $(p_2,\dots, p_k)$ is a permutation of
the sequence $(2,\dots,k)$
\\
 (the sequence
$(p_2,\dots, p_k)$ determines a partition of $(2,\dots,k)$ into two subsequences:
$(i_1<\dots <i_{\ell-1})$ such that all $p_{i_r}\leq \ell$
and $(j_1<\dots <j_{k-\ell})$ such that all  $p_{j_r} > \ell$,
these subsequences are used below);

\item[(iii)]
 the  sequence $(q_{i_1},\dots,q_{i_{\ell-1}})$  is an ordered $(\ell-1)$-element subset of
$\{2k-\ell+1, \dots, \widehat{q_1}, 
\linebreak
\dots, n\}$;

\item[(iv)]
the sequence  $(q_{j_1},\dots,q_{j_{k-\ell}})$  is a permutation of
the sequence $(k+1,\dots, 2k-\ell)$.

\end{enumerate}
The number of such $\al$ with fixed $q_1$ equals 
$(k-1)!\, (k-\ell)!\,\prod_{i=0}^{\ell-2}(n-2k+1+i)$.
For every such $\al$, we have
\bea
\frac{\der }{\der z_1^{(1)} }\der_{\{1,\dots,k\}}\der_{\{1,\dots,\ell, k+1,\dots,2k-\ell\}}  Q_\al(z)
 =  \frac{(-1)^{k-\ell}2^k}{z_{q_1}^{(1)}-z_{1}^{(1)}}\,.
\eea
It follows from \eqref{aell} and \eqref{2nd k} that
\bea
\frac{-a_\ell +a_{\ell-1}}{z^{(1)}_1 - z^{(1)}_{q_1}}\,=\,
\frac{(-1)^{k-\ell}2^k}{z_{q_1}^{(1)}-z_{1}^{(1)}}\,c_2\,
(k-1)!\, (k-\ell)!\,\prod_{i=0}^{\ell-2}(n-2k+1+i),
\eea
and  \eqref{Qfo} holds in this case.
Theorem \ref{thm 2nd} is proved.
\end{proof}

\subsection{Example } For $k=1$, we have
\bea
Q(z^{(1)}_1, \dots, z^{(1)}_n) = \frac {-1}2 \sum_{1\leq i<j\leq n}\ln(z^{(1)}_i-z^{(1)}_j)\,
(z^{(1)}_i-z^{(1)}_j)^2,
\eea
and Theorem \ref{thm 2nd} takes the following form.

\begin{thm}
\label{thm k=1} Let $i,j,m\in\{1,\dots,n\}$. Then
\bean
\label{k=1}
(\bar H_m(z^{(1)}_1, \dots, z^{(1)}_n)v_i,v_j) = 
\frac{\der^3 Q}{\der z^{(1)}_i \der z^{(1)}_j\der z^{(1)}_m}(z^{(1)}_1, \dots, z^{(1)}_n).
\eean

\end{thm}

Compare these formulas with formulas of Example \ref{ex1}.

\subsection{A relation between $P$ and $Q$}

\begin{thm}

Let $I, J$ be two $k$-element subsets of $\{1,\dots,n\}$. Then
\bean
\label{QtoP}
\frac 1{c_1}\,\der_I\der_J   P(z)\, 
\ =\  \frac 1{c_2}\,\sum_{m=1}^n \,z_m^{(1)}\,\frac{\der }{\der z_m^{(1)} } \der_I\der_J  Q(z)\,.
\eean

\end{thm}

\begin{proof}
The theorem follows from formulas \eqref{Pa} and 
 \eqref{Qa} for functions $P_\al(z)$ and $Q_\al(z)$ and the identity
$(x\frac{\der}{\der x}+ y\frac{\der}{\der y})\ln(x-y) = 1$.
\end{proof}

\begin{cor} 
The operator $\sum_{m=1}^n u_m\bar H_m(u_1,\dots,u_n)$ restricted to 
$\Sing W^{\ox n}[n-2k]$ is the scalar operator of multiplication by 
$-k(n-k+1)$.
\end{cor}

\begin{proof} Equation \eqref{QtoP} can be written as 
\bea
\frac 1{c_1} (v_I, v_J) = \frac1{c_2}
\Big(\sum_{m=1}^n z_m^{(1)}\bar H_m(z^{(1)}_1, \dots, z^{(1)}_n) v_I,v_J\Big).
\eea
Notice that $c_2/c_1=-k(n-k+1)$. This implies the corollary.
\end{proof}

\end{document}